\documentclass{article}
\usepackage[utf8]{inputenc}
\usepackage[a4paper,top=3cm,bottom=2cm,left=3cm,right=3cm,marginparwidth=1.75cm]{geometry}
\usepackage{amssymb,latexsym,amsmath,epsfig,amsthm} 
\usepackage{amsmath,amsthm,amssymb,amscd,esint,enumitem}
\usepackage[noblocks]{authblk}
\usepackage{mathrsfs}
\usepackage{tikz}
\usepackage{url}
\usepackage{xcolor}
\usepackage[colorlinks,citecolor=red,pagebackref,hypertexnames=false]{hyperref}

\usepackage{latexsym}
\usepackage{bm}
\usepackage{mathtools}

\makeatletter

\renewcommand\section{\@startsection {section}{1}{\z@}
{-30pt \@plus -1ex \@minus -.2ex}
{2.3ex \@plus.2ex}
{\normalfont\normalsize\bfseries\boldmath}}

\renewcommand\subsection{\@startsection{subsection}{2}{\z@}
{-3.25ex\@plus -1ex \@minus -.2ex}
{1.5ex \@plus .2ex}
{\normalfont\normalsize\bfseries\boldmath}}

\renewcommand{\@seccntformat}[1]{\csname the#1\endcsname. }

\makeatother

\newtheorem{theorem}{Theorem}[section]
\newtheorem{lemma}[theorem]{Lemma}
\newtheorem{proposition}[theorem]{Proposition}
\newtheorem{corollary}[theorem]{Corollary}

\theoremstyle{definition}

\newtheorem{definition}[theorem]{Definition}

\newtheorem{question}[theorem]{Question}

\newcommand{\R}{\mathbb{R}}

\newcommand{\C}{\mathbb{C}}

\newcommand{\ds}{\displaystyle}
\newcommand{\Aff}{\text{Aff}}
\newcommand{\SE}{\text{SE}_2(\R)}
\newcommand{\Stab}{\text{Stab}}

\title{A Structural Theorem for Sets With Few Triangles}
\author{Sam Mansfield, Jonathan Passant\thanks{Supported by the Heilbronn Institute.}}

\begin{document}

\maketitle

\begin{abstract}
We show that if a finite point set $P\subseteq \R^2$ has the fewest congruence classes of triangles possible, up to a constant $M$, then at least one of the following holds. 
\begin{itemize}
    \item There is a $\sigma>0$ and a line $l$ which contains $\Omega(|P|^\sigma)$ points of $P$. Further, a positive proportion of $P$ is covered by lines parallel to $l$ each containing $\Omega(|P|^\sigma)$ points of $P$.
    \item There is a circle $\gamma$ which contains a positive proportion of $P$.
\end{itemize}
This provides evidence for two conjectures of Erd\H{o}s. We use the result of Petridis-Roche-Newton-Rudnev-Warren on the structure of the affine group combined with classical results from additive combinatorics.
\end{abstract}

\section{Introduction}

Let $P$ be a finite point set in $\R^2$. Erd\H{o}s' famous distance conjecture \cite{E46} asks what the minimum number of distinct distances such a set can describe. This has lead to many beautiful techniques over 64 years culminating in the 2010 solution of Guth and Katz \cite{GK15}. 

Erd\H{o}s also posed much harder problems attempting to understand the structure of sets that determine few distinct distances. To provide context for these, suppose we have a square $\sqrt{N} \times \sqrt{N}$ lattice. A classical result of Landau \cite{L1909} on the growth of sums of squares shows that such a lattice gives $\Theta(N/\sqrt{\log N})$ distinct distances. One can find other lattices that give the same number of distances, but there are no known constructions which give fewer distances. We call a point set $P$ that has $c|P|/\sqrt{\log |P|}$ distances a \textit{near-optimal} point set. Erd\H{o}s' conjectures \cite{erdos1986structure} concern the structure of such near-optimal sets. The hardest such question is:
\begin{question}(Erd\H{o}s)
Do all \textit{near-optimal} point sets have a lattice structure?
\end{question}
Erd\H{o}s admitted that ``I really have no idea and the problem is perhaps too vaguely stated.''
Erd\H{o}s suggests that it would be nice to see if a near-optimal point set $P$ must contain at least $|P|^{1/2}$ points on a common line. This conjecture also appears too hard. The first bound due to Szemer\'edi (communicated by Erd\H{o}s \cite{erdos1975Line}) showed there must be a line containing at least $\sqrt{\log|P|}$ points of $P$. Using K\H ov\'ari–S\'os–Tur\'an \cite{kovari1954problem} this can be improved to at least $\log|P|$ points, see \cite{SheffBlogLine}. There is a weaker version of Erd\H{o}s' line question that is still open.
\begin{question}\label{Question: PoorLine}(Erd\H{o}s)
Prove or disprove: For a sufficiently small $\varepsilon>0$, every near-optimal point set $P$ contains at least $c|P|^\varepsilon$ points on a common line.
\end{question}
Lund, Sheffer and de Zeeuw \cite{lund2016bisector} used bisector energy to show for any $0<\sigma\leq 1/4$ there is either a line or circle containing $c|P|^\sigma$ points of $P$, or there are $c|P|^{8/5-12\sigma/5-\varepsilon}$ lines each containing at least $c\sqrt{\log|P|}$ points of $P$.
There has been recent work on the converse problem: lines, circles and constant degree polynomials cannot have too large an intersection with a near-optimal set \cite{sheffer2016fewdistancesNoLineCircles, raz2015fewdistancesNoLineCircles, pach2017distinct}. By combining these three results, we obtain that for every near-optimal set $P$, every constant-degree algebraic curve contains at most $c|P|^{43/52}$ points of $P$.

One possible approach, suggested by Nets Katz, is to show that the additive energy of a near-optimal point set is large \cite[Problem 34]{sheffer2014openproblems}. Following this philosophy, Hanson \cite{hanson2018structure}, Roche-Newton \cite{rocheNewton2016structure} and Pohoata \cite{pohoata2019structure} demonstrate that near-optimal Cartesian products have small difference sets.

Since the introduction of the Guth-Katz--Elekes-Sharir framework \cite{ES11, GK15}, it has been productive to view distinct distances as congruence classes of pairs of points $(p,q)\in P^2$ under the action of the group of rigid motions on $\R^2$. This was essential to the Guth-Katz \cite{GK15} result which showed a point set $P$ has $\Omega\left(|P|/\log|P|\right)$ distinct distances. One can think of classes of congruent triangles as the congruence classes of triples of points $(p,q,r)$ under the same action (see Figure \ref{fig:CongTriangles}). This perspective was used by Rudnev \cite{R12} to show that a point set $P$ describes at least $\Omega(|P|^2)$ distinct classes of congruent triangles. 

\begin{figure}[ht]
    \centering
    \includegraphics[scale=0.4]{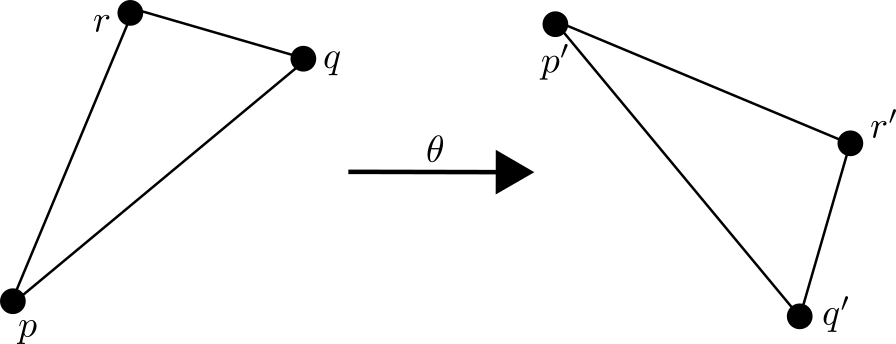}
    \caption{A pair of congruent triangles, with a rigid motion $\theta$ transforming $(p,q,r)\xrightarrow{\theta}(p',q',r')$.}
    \label{fig:CongTriangles}
\end{figure}

Like with distances, we can see that Rudnev's lower bound is sharp by looking at the $\sqrt{N} \times \sqrt{N}$ integer lattice. There are two further examples for triangles: points in an arithmetic progression on a line; or vertices of a regular polygon (see Figure \ref{fig: LineAndPolygon}).

\begin{figure}[ht]
    \centering
    \includegraphics[scale=0.4]{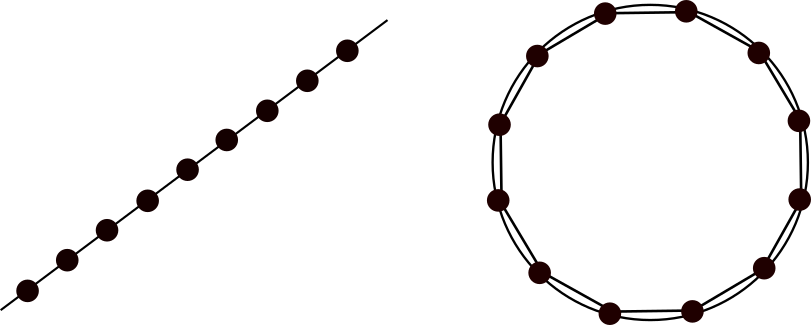}
    \caption{An arithmetic progression on a line and points at the vertices of a regular polygon.}
    \label{fig: LineAndPolygon}
\end{figure}

One can also generalise these examples, by repeating them $O(1)$ times. If $P$ is on $O(1)$ parallel lines with the same arithmetic progression on each then $P$ defines $\Theta(|P|^2)$ congruence classes of triangles. Similarly, if $P$ is on $O(1)$ concentric circles with points lying at a vertex of a scaled version of the same polygon, then we have few classes of congruent triangles.

We consider point sets with few classes of congruent triangles. We can prove much stronger structural results in this setting than those known for distances. Before we describe our result we formalise our assumption. Rudnev's sharp lower bound gives us the following definition for triangle-near-optimal point sets.
\begin{definition}
We say a finite point set $P$ in $\R^2$ is \textit{triangle-near-optimal} if, for some constant $M$, $P$ describes at most $M|P|^2$ classes of congruent triangle. 
\end{definition}
We show that all triangle-near-optimal sets are similar to the above examples in two ways: Theorem \ref{Thm: MainStructuralResult} gives us the geometric similarity. Corollary \ref{Coro: FewTrianglesLargeEnergy} shows that all triangle near optimal sets have strong additive or multiplicative structure, similar to the above examples.
\begin{theorem}\label{Thm: MainStructuralResult}
Suppose we have a finite point set $P$ in $\mathbb{R}^2$. Let $M$ be a positive constant and let $c, c'$ and $C$ be positive constants that depend only on $M$. If $P$ contains $M|P|^2$ classes of congruent triangles (triangle-near-optimal) then either
\begin{itemize}
    \item There is a line $l$ which contains $c|P|^\sigma$ points of $P$ for some $0<\sigma\leq 1$. Further, a positive proportion of $P$ is covered by lines parallel to $l$ each containing $c'|P|^\sigma$ points of $P$.
    \item There is a circle $\gamma$ which contains $C|P|$ points of $P$. 
\end{itemize}
\end{theorem}
We also give sharp energy bounds for such sets, see the slightly more general Theorem \ref{Thm: RichActionsGiveStructure} below.

Theorem \ref{Thm: MainStructuralResult} solves Question \ref{Question: PoorLine} of Erd\H{o}s for sets with few classes of congruent triangles. Unfortunately, while triangles and distances are similar from the point of view of rigid motions, we cannot prove a direct relationship between sets with few triangles and those with few distances. This leads to the following question, a positive answer to which would fully resolve Question \ref{Question: PoorLine} of Erd\H{o}s.
\begin{question}
Let $P$ be a finite point set in $\R^2$ with at most $c|P|/\sqrt{\log|P|}$ distinct distances. Does $P$ describe at most $M|P|^2$ classes of congruent triangles for some constant $M$? The constant $M$ may depend on $c$.
\end{question}
We will discuss how our methods fail to answer this question after describing the proof of Theorem \ref{Thm: MainStructuralResult} (see Question \ref{Question: ActionDistn}).
The proof of Theorem \ref{Thm: MainStructuralResult} consists of several steps combining the Elekes-Sharir--Guth-Katz framework and structure of the affine group with the traditional tools of Balog-Szemer\'edi-Gowers and the Freiman-type result of Green and Ruzsa from additive combinatorics. Before we outline the proof, we need to introduce some useful notions.

Let $P$ be a finite point set in $\R^2$, $M$ a positive constant (constants do not grow with $|P|$). We consider the set of classes of congruent triangles determined by $P$, defined as 
\begin{equation*}\label{Def: TriangleSet}
     T(P) = \{ (|p-q|,|q-r|,|p-r|): p,q,r \in P \}.
\end{equation*}
To apply the tools from additive combinatorics, we interpret additive structure in $P$ using the additive energy of $P$. The additive energy is defined as
\begin{equation*}
E_+(P)=|\{(p_1,p_2,p_3,p_4)\in P^4:p_1+p_2=p_3+p_4\}|.
\end{equation*}
To interpret multiplicative structure, we identify $\R^2$ with $\C$ in the usual way, so that we can think of $P$ as a subset of $\C$. This allows us to define the multiplicative energy of $P$ as 
\begin{equation*}
E_\times(P)=|\{(p_1,p_2,p_3,p_4) \in P^4 \subset \C^4 : p_1p_2=p_3p_4\}|.
\end{equation*}
This identification also allows us to interpret the group of rigid motions on $\R^2$, which we denote by $\SE$, as a subgroup of the affine group on $\C$, denoted $\Aff(\C)$. For technical calculations we use the isomorphism $ \Aff(\C) \cong \C^\times \ltimes \C$. For details of this isomorphism, and the action of $\Aff(\C)$ on $\C$ see the start of Section \ref{Sec: LargeGroupEnergyGivesLargePointEnergy}.
We call a given rigid motion $\theta \in \SE$ $k$-rich (with respect to the point set $P$) if
$$ |P\cap \theta P| = k.$$

Finally, we introduce group energy, which plays a central role in the proof. Let $S$ be a subset of a multiplicative group $G$. Then we define its group energy as 
\begin{equation*}\label{def: ElekesEnergy}
    E(S) = |\{ (g_1,g_2,g_3,g_4) \in S^4: g_1g_2^{-1}=g_3g_4^{-1}\}|.
\end{equation*}
Throughout the paper we should think of $S$ as a set of rigid motions within the affine group, so $S \subseteq \SE \leq \Aff(\C)$. Throughout, we use $\theta$ to denote an element of $\SE$, a rigid motion. We will use $g$ and $h$ to denote more general elements of the affine group.

The proof is broken down into four main steps, a rough outline is the following.
\begin{enumerate}
    \item[\textbf{Step 1}] (Few triangles give many rich actions) If $P$ has $M|P|^2$ congruence classes of triangles then there are roughly $|P|$ members of $\SE$ which are $|P|$-rich.
    \item[\textbf{Step 2}] (Many rich actions lead to a large group energy) If $S$ is the large set of rich rigid motions, then its group energy is as large as possible.
    \item[\textbf{Step 3}] (Rich symmetries with large group energy gives $P$ large energy) Large group energy gives large intersection between very-rich rigid motions and a `nice' coset of $\Aff(\C)$. This large intersection means either the additive or multiplicative energy of $P$ is large.
    \item[\textbf{Step 4}] (Structure in the affine group gives $P$ structure) We use the large coset intersection from Step 3. We obtain the circle structure using a coset of the affine group directly. For the line structure, we have to pass through the additive energy using Balog-Szemer\'edi-Gowers and Freiman type results.
\end{enumerate}
The first and second steps require some delicacy to ensure that there are no log-type losses, see Lemma \ref{Lem: ManyVeryRichActions}. The third step is an application of a Balog-Szemer\'edi-Gowers-type result in the affine group due to Petridis-Roche--Newton-Rudnev-Warren \cite{PRNRW22}. Step four is divided into a multiplicative and additive cases. The multiplicative case uses only the structure of cosets of $\SE$. The additive case we require the application of Balog-Szemer\'edi-Gowers and a result of Mudgal \cite{Akshat19} that relies on the Freiman-type result of Green-Ruzsa \cite{green2007freiman}.

We summarise steps 2 to 4 in the following Theorem.
\begin{theorem}\label{Thm: RichActionsGiveStructure}
Suppose that $P$ is a point set in $\R^2$, let $C_1$ and $C_2$ be positive constants. Suppose there is a set $S$ in $\SE$ of $C_1|P|$ actions, each of which are $(C_2|P|)$-rich with respect to the points in $P$. Then at least one of the following holds:
\begin{enumerate}
    \item 
    There is a sharp energy bound $\ds E_+(P) \geq C_3|P|^3$. This energy bound provides the structure:\\
    There is a line $l$ which contains $\Omega(|P|^\sigma)$ points of $P$ for some $0<\sigma\leq 1$. Further, a positive proportion of $P$ is covered by lines parallel to $l$ each containing $\Omega(|P|^\sigma)$ points of $P$.

    \item There is a $t \in \R^2$ such that the sharp energy bound $\ds E_\times(P-t)\geq C_4|P|^3$ holds. Using the affine group cosets directly we can show:\\
    There is a circle $\gamma$ which contains $\Omega(|P|)$ points of $P$.
        
\end{enumerate}
\end{theorem}
To ensure readability of the proof we will focus on the constants seen in the triangle case only. Theorem \ref{Thm: RichActionsGiveStructure} can be proved by using the constants $C_1$ and $C_2$ in place of those provided by Corollary \ref{Coro: SetWithLargeEnergy} in Sections \ref{Sec: RichActionsGiveLargeEnergy}, \ref{Sec: LargeGroupEnergyGivesLargePointEnergy} and \ref{Sec: StructureInP}.

The energy bounds gained here are enough--under the triangle assumption--to resolve the energy problem of Katz \cite[Problem 34]{sheffer2014openproblems}. 

We emphasise that the bound on the additive energy $E_+(P)$ was necessary to demonstrate line structure, however the bound on the multiplicative energy $E_\times(P)$ is not required to demonstrate circle structure. Once we rely on the energy to get structure, there are results of Stanchescu \cite{stanchescu2002planar} and Erd\H{o}s-F\"uredi-Pach-Ruzsa \cite{gridRevisited} that show that we need essentially-optimal energy bounds to have any chance of structure. We discuss this further in Section \ref{Sec: AdditiveStructure}.

For distances, step 1 fails. Theorem \ref{Thm: RichActionsGiveStructure} requires at least $c_1|P|$ actions each being $(c_2|P|)$-rich. For a set with at most, say, $|P|/\log|P|$ distinct distances (which is near minimal) our methods cannot guarantee many $|P|$-rich actions. Indeed, it is not clear that a minimal distance set must contain many very rich actions.
When looking at distances, unlike the triangle case, there cannot be examples where all actions are $(c_2|P|)$-rich (as such sets have to have at least $|P|$ distances). This leads to the following question.
\begin{question}\label{Question: ActionDistn}
If $P$ is a finite point set in $\R^2$ with $o(|P|)$ distinct distances, can one say anything about the distribution of the $k$-rich actions for $k=|P|^\sigma$, $0< \sigma\leq 1$?
\end{question}
It is a quick calculation using the Guth-Katz bound \cite[Proposition 2.5]{GK15} to show that $k$-rich actions with $k\leq \log|P|$ or $k\geq |P|/2$ cannot provide a large enough energy for a set with such few distances, so some sort of polynomially-rich actions are necessary. We can show that there is at least one $\sigma$ in the range $0<\sigma \leq 1$ which achieves the Guth-Katz bound. Our method is not quantitative, so the best group energy bound we can get is $E(S) = \Omega( |S|^{2+1/3} )$. One would need an exponent better than $2+1/2$ for our methods to give any structure in $P$.

In the square lattice example, each $k$ has the maximum number of $k$-rich actions i.e. for all $k$ we have $|S_{\geq k}| = \Theta(|P|^3k^{-2})$. We do not have enough evidence to conjecture that such a distribution holds across all near-optimal sets. Any examples disproving this would be of great interest. We also note that our methods gives better energy bounds on $P$ as the richness of the group actions increases. For the structural results it is necessary for essentially-optimal energy bounds, so we can afford no polynomial loss in the richness of our actions.
\subsubsection*{Acknowledgements}
The authors would like to thank Misha Rudnev and Oleksiy Klurman for suggesting the problem and for helpful discussions. The second author would like to thank Akshat Mudgal for a very helpful conversation and for pointing out \cite[Theorem 1.5]{Akshat19}. We would like to thank the Heilbronn Institute for Mathematical Research for supporting the Focused Research Workshop ``Testing Additive Structure'' where we received many helpful comments and input. We would especially like to thank Brandon Hanson for pointing out the examples in \cite{gridRevisited} and \cite{stanchescu2002planar}. The authors would like the thank the anonymous
referees for a careful reading of the manuscript and helpful comments.

\section{Few Triangles Give Many Rich Actions}\label{Sec: FewTrianglesMAnyRichActions}

In this section we detail the first step of our proof: That sets with few classes of congruent triangles have many very-rich actions.
For a point set $P$ and a positive integer $k$, recall that an element $\theta$ in $\SE$ is a $k$-rich rigid motion if $|P \cap \theta P| = k$. Define $S_{\geq k}=S_{\geq k}(P)$ to be the elements of $\SE$ that are at least \(k\)-rich rigid motions of $P$. 
We state the incidence result of Guth-Katz \cite[Proposition 2.5]{GK15} with an explicit constant $C$ as
\begin{equation}\label{Eq: GKExplicitConstant}
    |S_{\geq k}(P)| \leq C\frac{|P|^3}{k^2}.
\end{equation}
We only need that $C$ is finite, see \cite{kollar2015GuthKatz} for explicit constants. Using this we can prove the following.

\begin{lemma}\label{Lem: ManyVeryRichActions}
Suppose we have a point set $P$ in $\R^2$ with exactly $M|P|^2$ congruence classes of triangles for some constant $M$. Let $C$ be the Guth-Katz constant. Then we have at least $(3CM)^{-3}|P|$ values of $k$ such that both the following hold:
\begin{itemize}
    \item $\ds |P| \geq k\geq |P|/3CM$, and
    \item $\ds |S_{\geq k}(P)| \geq \frac{|P|}{3M}.$
\end{itemize}
\end{lemma}
When we use this lemma later we will only need one such value of $k$, however it is easier to prove the above.

The proof of Lemma \ref{Lem: ManyVeryRichActions} uses the initial step of the Elekes-Sharir--Guth-Katz framework. We recall the necessary steps prior to beginning the proof.
\subsection{Counting Energy Using Rich Rigid Motions}
For a point set $P$ define its \textit{triangle energy} as
$$ E_T(P) = |\{ (p,q,r,p',q',r') \in P^6 : \text{ tri. } pqr \text{ congruent to tri. } p'q'r'\}|.$$
We aim to rephrase this energy in terms of $k$-rich rigid motions. We can then use Cauchy-Schwarz and the fact that we have very few classes of congruent triangles to give an essentially-optimal lower bound. See the forthcoming equation \eqref{Eq: RigidMotionsIneq}.

We note the earlier observation of Rudnev \cite{R12}, that two triangles are congruent if and only if there is a rigid motion taking one to the other. See Figure \ref{fig:CongTriangles}. Thus, we can count the triangle energy by counting the number of rigid motions weighted by how rich these motions are. Indeed, note that if $\theta$ is $k$-rich then there are ${k\choose 3}$ triples of the form $(p,q,r,\theta p,\theta q, \theta r)$.
Letting $S_{=k}$ be the set of exactly $k$-rich rigid motions we can count the triangle energy as
$$ E_T(P) = \sum_{k=3}^{|P|}|S_{=k}|{k \choose 3}. $$
We change this sum to use $S_{\geq k}$, the number of at least $k$-rich rigid motions, using that $|S_{=k}| = |S_{\geq k}| - |S_{\geq k+1}|$. One then notes that the term $|S_{\geq k}|$ occurs in the sum with weight
$$ {k \choose 3} - {k-1 \choose 3} \leq k^2.$$
Thus we have that
$$ E_T(P) \leq \sum_{k=3}^{|P|}|S_{\geq k}|k^2.$$

To obtain the lower bound on this energy, we use Cauchy-Schwarz. For a triangle class $t$ in $T(P)$ define $r(t)$, its number of realisations in $P$, as
$$ r(t) = |\{ (p,q,r) \in P^3 : \text{ triangle } pqr \text{ is congruent to } t \}|. $$
We then have that
\begin{equation}\label{Eq: CauchySchwarzStep}
    |P|^6 = \left( \sum_{t\in T(P)} r(t)\right)^2 \leq |T(P)|\sum_t r^2(t).
\end{equation}
Notice that the final sum here is the same as the size of six-tuples $(p,q,r,p',q',r')$ where the triangles $pqr$ and $p'q'r'$ are congruent. So this is exactly the triangle energy.

Using this, along with the earlier energy upper bound, one has
\begin{equation}\label{Eq: RigidMotionsIneq}
    |P|^6 \leq |T(P)|\sum_{k=3}^{|P|} |S_{\geq k}|k^2.
\end{equation}
This is the bound we will need for the remainder of the section. For completion, note that Rudnev's lower bound follows by the application of the Guth-Katz bound, see \eqref{Eq: GKExplicitConstant}. This gives
\begin{equation}\label{Eq: RudnevTriangle}
    \sum_{k=3}^{|P|} |S_{\geq k}|k^2 \leq C|P|^4.
\end{equation}
\subsection{Proof of Lemma \ref{Lem: ManyVeryRichActions}}
With \eqref{Eq: RigidMotionsIneq} established, we are ready to begin the proof of Lemma \ref{Lem: ManyVeryRichActions}.
\begin{proof}[Proof of Lemma \ref{Lem: ManyVeryRichActions}]

As our $k$ represent the richness of rigid motions acting on $P$, it is natural to only consider $k$ in the range $2 \leq k \leq |P|$. Indeed, we need $k \geq 2$ to ensure $S_{\geq k}$ is finite and for all $\theta$ we have $|P\cap \theta P| \leq |P|$. We call a $k$ good if both 
\begin{itemize}
    \item $\ds |P| \geq k\geq |P|/3CM$, and
    \item $\ds |S_{\geq k}(P)| \geq \frac{|P|}{3M}.$
\end{itemize}
Let $X$ be the number of good $k$. The bad $k$ are in one of two cases.
In the first case $k\geq |P|/3CM$ and $|S_{\geq k}(P)|$ is smaller than the threshold. We call such values of $k$ $1$-bad with their total number being $Y_1$. The second case is when $k<|P|/3CM$ with no assumption on $|S_{\geq k}|$. We call such values of $k$ $2$-bad with their total number being $Y_2$. These cases partition all values of $k$ in the range $2\leq k \leq |P|$, so $X+Y_1+Y_2=|P|-1$.
We suppose, for contradiction, that $X < |P|/27M^3C^3$.

As our set has $M|P|^2$ triangles, \eqref{Eq: RigidMotionsIneq} gives us that
\begin{equation}
    \frac{|P|^4}{M} \leq \sum_{k=1}^{|P|} |S_{\geq k}|k^2.
\end{equation}
Notice that \eqref{Eq: RudnevTriangle} shows that the lower bound here is essentially the best possible.

We split the sum up into the sum over the good $k$ and the sum over the two bad sets. For the sum of the good values of $k$ we use (\ref{Eq: GKExplicitConstant}). With $k \geq |P|/3MC$ this gives $|S_{\geq k}| \leq 9M^2C^3|P|$. Thus, using $k\leq |P|$ we have,
$$  \sum_{k \text{ good}} |S_{\geq k}|k^2 \leq X\cdot9M^2C^3|P| \cdot |P|^2 < \frac{|P|^4}{3M}.$$
With the last estimate using our assumption that $X < |P|/27M^3C^3$. 

We now estimate the contribution from the $1$-bad $k$. As each term is $1$-bad we know that
$$ |S_{\geq k}| < \frac{|P|}{3M}.$$
Using this, and again $k\leq |P|$, we have that
$$ \sum_{k~1\text{-bad}} |S_{\geq k}|k^2 \leq Y_1\cdot\left(\frac{|P|}{3M}\right)\cdot |P|^2. $$
To conclude this case it suffices to show that
$$  Y_1\left(\frac{|P|^3}{3M}\right) < \frac{|P|^4}{3M}.$$
This follows as $k$ can only exist in the range $2 \leq k \leq |P|$, so the number of 1-bad such $k$ is strictly less than $|P|$ i.e. $Y_1 \leq |P|-1 < |P|$.

We now have to deal with the $2$-bad values of $k$. Being 2-bad means that $k < |P|/3CM$. We use (\ref{Eq: GKExplicitConstant}) to bound each $|S_{\leq k}|$, obtaining
$$ \sum_{k~2\text{-bad}} |S_{\geq k}|k^2 \leq Y_2\cdot C|P|^3 < \frac{|P|^4}{3M}.$$
The final inequality using $Y_2 < |P|/3CM$, as this is the range of $k$ in the 2-bad case.
Thus, under the assumption that $X < |P|/27M^3C^3$, we have that
$$ \frac{|P|^4}{M} \leq \sum_{k=1}^{|P|} |S_{\geq k}|k^2 < \frac{|P|^4}{M},$$
this contradiction completes the proof.
\end{proof}

\section{Many Rich Actions Give Large Group Energy}\label{Sec: RichActionsGiveLargeEnergy}
We show that if we have very few classes of congruent triangles then there must be a large set of rigid motions that has very large group energy. See the forthcoming Corollary \ref{Coro: SetWithLargeEnergy}. The proof of this relies on a technical inequality which we state as a lemma.
\begin{lemma}\label{Prop: LargeGroupEnergy}
Let $S_{\geq k}$ be a set of at least $k$-rich rigid motions in $\SE$. Then
$$ \frac{k^6|S_{\geq k}|^4}{C|P|^7} \leq E(S_{\geq k}). $$
\end{lemma}
We than combine Lemma \ref{Prop: LargeGroupEnergy} with one of the good $k$ guaranteed by Lemma \ref{Lem: ManyVeryRichActions}. This gives us the following corollary.

\begin{corollary}\label{Coro: SetWithLargeEnergy}
If $P$ is a set with at most $M|P|^2$ classes of congruent triangles, then there is a set $S\subset \SE$ such that all of the following hold
\begin{itemize}
    \item $\ds |S| \geq \frac{|P|}{3M}.$
    \item Each element of $S$ is at  least $\ds \left(\frac{|P|}{3CM}\right)$-rich when acting on $P$.
    \item $ \ds \frac{|S|^3}{(3CM)^7} \leq E(S). $
\end{itemize}
\end{corollary}
We leave the more technical proof of Lemma \ref{Prop: LargeGroupEnergy} until after the short proof of Corollary \ref{Coro: SetWithLargeEnergy}. 
\begin{proof}[Proof of Corollary \ref{Coro: SetWithLargeEnergy}]
Lemma \ref{Lem: ManyVeryRichActions} gives us $|P|/(3CM)^3$ values of $k$ such that $|S_{\geq k}| \geq \frac{|P|}{3M}$ where $k \geq |P|/3CM$. We take one such $k$ and let $S=S_{\geq k}$. By Lemma \ref{Prop: LargeGroupEnergy} we have
$$ \frac{k^6|S|^4}{C|P|^7} \leq E(S). $$
Using the estimate on $k$ six times and the estimate on $|S|$ once we see that $\frac{|S|^3}{(3CM)^7} \leq E(S)$.
\end{proof}
We finish the section by proving Lemma \ref{Prop: LargeGroupEnergy}.
\begin{proof}[Proof of Lemma \ref{Prop: LargeGroupEnergy}]
We now use the methodology introduced and developed by Elekes in \cite{elekes1997linear,elekes1997number,elekes1998linear,elekes2001sums}. Let $P(x)$ be used as the indicator function of $x \in P$ and let $S=S_{\geq k}$. Using that each element $\theta$ of $S$ is at least $k$-rich we have
\begin{align*}
    k|S|   \leq \sum_{p \in P}\sum_{\theta\in S}P(\theta p).
\end{align*}
Using Cauchy-Schwarz gives
$$ k|S|\leq |P|^{1/2} \left(\sum_{p \in P} \sum_{\theta, \phi} P(\theta p)P(\phi p) \right)^{1/2}.$$
As we have $p\xrightarrow{\theta}\theta p$ and $p\xrightarrow{\phi}\phi p$, we can see that $\theta p \xrightarrow{\phi\theta^{-1}} \phi p$.
Thus, we can relabel the sum over $\theta, \phi$ as the sum over $\varphi=\phi\theta^{-1}$ in $SS^{-1}$. When we do this we have to count the repeated representations of $\varphi$ using the weight
$$r_{SS^{-1}}(\varphi) = |\{ (\theta, \phi) \in S \times S: \phi\theta^{-1} = \varphi\}|.$$
So letting $p' = \theta p$ this allows us to write
$$\sum_{p \in P} \sum_{\theta, \phi} P(\theta p)P(\phi p) = \sum_{p'\in P} \sum_{\varphi} P(\varphi p')r_{SS^{-1}}(\varphi). $$
Squaring both sides, we see that
\begin{align*}
    k^2|S|^2 &\leq |P| \sum_{p'\in P} \sum_{\varphi} P(\varphi p')r_{SS^{-1}}(\varphi),\\
    k^2|S|^2 &\leq |P| \sum_{\varphi} r_{SS^{-1}}(\varphi) \sum_{p'\in P} P(\varphi p').
\end{align*}
We apply H\"older to obtain
\begin{align*}
    k^2|S|^2 &\leq |P| \left(\sum_{\varphi} r_{SS^{-1}}^{3/2}(\varphi)\right)^{2/3}\left( \sum_{\varphi} \sum_{p,p',p''\in P}P(\varphi p)P(\varphi p')P(\varphi p'') \right)^{1/3}.
\end{align*}
The second sum is exactly the triangle energy discussed in Section \ref{Sec: FewTrianglesMAnyRichActions}. Indeed, we are counting the size of the set
$$ E_T(P) =  \{ (p,p',p'', \varphi p, \varphi p', \varphi p'') \in P^6 : \varphi \in \SE \},$$
which, as in \eqref{Eq: RudnevTriangle}, is bounded by $C|P|^4$. Cubing both sides, we have 
\begin{align*}
        k^6|S|^6 &\leq |P|^3 \left(\sum_{\varphi} r_{SS^{-1}}^{1/2}(\varphi)r_{SS^{-1}}(\varphi)\right)^{2}C|P|^4.
\end{align*}
We then apply Cauchy-Schwarz,
\begin{align*}
    k^6|S|^6 &\leq |P|^3 \left(\sum_{\varphi} r_{SS^{-1}}(\varphi)\right)\left(\sum_{\varphi}r_{SS^{-1}}^{2}(\varphi)\right)C|P|^4.
\end{align*}
We note that the first sum is just $|S|^2$ and the second is the group energy $E(S)$. Both of these calculations are the same as the ones in Section \ref{Sec: FewTrianglesMAnyRichActions}, after \eqref{Eq: CauchySchwarzStep}. Rearranging, we have the desired bound
$$ \frac{k^6|S|^4}{C|P|^7} \leq E(S). $$
\end{proof}
\section{Rich Symmetries With Large Group Energy Give $P$ Large Energy}\label{Sec: LargeGroupEnergyGivesLargePointEnergy}
We demonstrate the structure that a set $P$ with few classes of congruent triangles has in the affine group. See the forthcoming Corollary \ref{Coro: LargeCosetIntersection}. We also show that point sets with few classes of congruent triangles demonstrate sharp energy bounds. See Corollary \ref{Coro: FewTrianglesLargeEnergy}.

We will make extensive use of the affine group $\Aff(\C)$. We also make the usual identification of \(\mathbb{R}^2\) with \(\mathbb{C}\). We use that \(\text{Aff}(\mathbb{C})=\mathbb{C}^\times \ltimes \mathbb{C}\) with identity $(1,0)$ and semidirect product multiplication
\begin{equation}\label{eq: AffGroupProduct}
    (g_1,g_2)(h_1,h_2) = (g_1h_1, g_1h_2 + g_2).
\end{equation}
The group of rigid motions $\text{SE}_2(\mathbb{R})$ becomes a subgroup of \(\text{Aff}(\mathbb{C})=\mathbb{C}^\times \ltimes \mathbb{C}\) via embedding it as the subgroup $S^1 \ltimes \C$. Let $x=(x_1,x_2)$ and $t=(t_1,t_2)$ be points in $\R^2$, the embedding is as follows
\begin{align}\label{eq: Embedding1}
    \text{Rotation by angle } \theta \text{ about the point } x &\to (e^{i\theta}, (x_1+ix_2)(1-e^{i\theta}))\\\
    \text{Translation by } t &\to (1, t_1 + it_2)\label{eq: Embedding2}
\end{align}
One can check this is a group isomorphism.
This helps by letting us exploit the geometry of the affine group, similar approaches can be found be found in \cite{RS18, PRNRW22}. The affine group acts on $\C$, in particular on $P \subseteq \C$, via the action 
\begin{equation}\label{eq: GroupAction}
    (g_1,g_2)\cdot x = g_1 x + g_2.
\end{equation}
We will care about two types of subgroups of $\Aff(\C)$. In particular those that can be thought of as lines when $\text{Aff}(\C)$ is identified as the `plane' $\C^2\setminus \{(0,z): z \in \C\}$. 
\begin{itemize}
    \item The unipotent subgroup $U_0=\{(1,z) : z \in \C\}$ is a vertical line through the identity. A coset $gU_0$ is the vertical line though $g$.
    \item The maximal tori $T(z)$, these are the stabiliser subgroups under the action in \eqref{eq: GroupAction}. So, for $z\in \C$, we define $T(z)=\Stab(z)$. Tori correspond to non-vertical lines through the identity $(1,0)$. A coset $gT(z)$ is a non-vertical line passing through $g$.
\end{itemize}
 We use the following Theorem of Petridis-Roche--Newton-Rudnev-Warren \cite{PRNRW22}, the proof of which is an application of Rudnev's point-plane bound \cite{rudnev2018pointPlane}. The version stated below is adapted to our setting, for a positive characteristic version see \cite{PRNRW22}.
 Similarly to the group energy we define, for $S$ any subset of a group, the energy $E^*(S) = |\{ (g_1,g_2,g_3,g_4) \in S^4: g_1g_2=g_3g_4\}|$.

\begin{theorem}\label{Thm: PoorLinesImpliesSmallEnergy}
Let $S$ be a finite set of transformations in the affine group $\text{Aff}(\mathbb{C})$ such that no non-vertical line
contains more than $H$ points of
$S$, and no vertical line contains more than $V$ points of $S$. Then,

$$\max\{E(S), E^*(S)\} = O\left( V^{1/2}|S|^{5/2} + H|S|^2 \right).$$
\end{theorem}
\noindent Shkredov \cite{shkredov2021modular} shows that $E^*(S)\leq E(S)$, so we can use $E(S)$ as this maximum.

Using Corollary \ref{Coro: SetWithLargeEnergy} with Theorem \ref{Thm: PoorLinesImpliesSmallEnergy} we find a rich line in $\Aff(\C)$. Indeed, as a direct consequence of the group energy bound for the set $S$ in Corollary \ref{Coro: SetWithLargeEnergy} we can say that there are positive constants $c_1$ and $c_2$ such that one of the following must exist:
\begin{itemize}
    \item A vertical line in $\Aff(\C)$ that contains at least $\ds c_1\left(\frac{|S|}{(3CM)^{14}}\right)$ points of $S$;
    \item A non-vertical line in $\Aff(\C)$ that contains at least $\ds c_2\left(\frac{|S|}{(3CM)^7}\right)$ points of $S$.
\end{itemize}
Corollary \ref{Coro: SetWithLargeEnergy} tells us that $|S|\geq|P|/3M$, this gives us the following corollary.
\begin{corollary}\label{Coro: LargeCosetIntersection}
Let $P$ be a point set in $\R^2$ with at most $M|P|^2$ classes of congruent triangles. Let $C$ be the constant in the Guth-Katz theorem. Corollary \ref{Coro: SetWithLargeEnergy} guarantees a set $S$ which contains at least $|P|/3M$ rigid motions, all of which are at least $(|P|/3CM)$-rich. For this set $S$, there are positive constants $c_1,c_2$ and $z\in \C$ such that at least one of the following holds:
\begin{itemize}
    \item There is some $g$ in $\Aff(\C)$ such that  $\ds |gU_0 \cap S| \geq c_1\frac{|P|}{(3CM)^{14}(3M)}$.
    \item There is some $g$ in $\Aff(\C)$ and $z$ in $\C$ such that $\ds |gT(z) \cap S| \geq c_2\frac{|P|}{(3CM)^7(3M)}$.
\end{itemize}
\end{corollary}
Unfortunately, there are point sets $P$ where both the conclusions of Corollary \ref{Coro: LargeCosetIntersection} can be achieved simultaneously\footnote{For example: Let $P_1$ be an arithmetic progression on a line and $P_2$ be the set of vertices of a regular polygon. Suppose that $|P_1|=|P_2|$. Then $P=P_1\cup P_2$ we would have lines/cosets of both types.}.
In Section \ref{Sec: MultiplicativeStructure} we show that we can take \(g\) and our cosets in the subgroup $\SE = S^1\ltimes \C$. This will be important for establishing the circle structure.

We can also prove the following result about the energy of the point set $P$. Recall that $E_\times(P)$ is the multiplicative energy, treating $P$ as a set of complex numbers. 
\begin{proposition}\label{Prop: LargePEnergy}
Suppose that $P$ is a point set in $\R^2$, let $C_1, C_2$ and $C_3$ be positive constants. Suppose there is a set $S$ in $\SE$ of rigid motions each of which are $(C_1|P|)$-rich when acting on $P$. Both of the following hold.
\begin{itemize}
     \item If $\ds |gU_0 \cap S| \geq C_2|P|$ then $E_+(P)\geq C_2C_1^3|P|^3$,
     \item If $\ds |gT(z) \cap S| \geq C_3|P| $ then there is some $t\in \C$ such that $ E_\times(P-t)\geq C_3|P|(C_1|P|-1)^2$.
\end{itemize}
\end{proposition}
Using the set $S$ from Corollary \ref{Coro: SetWithLargeEnergy} with the additional properties from Corollary \ref{Coro: LargeCosetIntersection}, Proposition \ref{Prop: LargePEnergy} allows us to conclude the following.
\begin{corollary}\label{Coro: FewTrianglesLargeEnergy}
Let $P$ be a point set in $\R^2$ with at most $M|P|^2$ classes of congruent triangles. Then at least one of the following is true:
\begin{itemize}
     \item $\ds E_+(P)\geq \frac{|P|^3}{(3CM)^{16}(3M)}$,
     \item There exists some $t\in \C$ such that $\ds E_\times(P-t)= \Omega \left(\frac{|P|^3}{(3CM)^{9}(3M)}\right)$.
\end{itemize}
\end{corollary}

The remainder of the section is devoted to the proof of Proposition \ref{Prop: LargePEnergy}.
\begin{proof}[Proof of Proposition \ref{Prop: LargePEnergy}]
We prove the additive statement first. We include the multiplicative proof too as there are technical differences that need be checked. For notational ease let $k=C_1|P|$ throughout the proof.

For the additive case we note that the assumptions give us a set $S$ such that
\begin{equation}\label{eq: BkDefinition}
    |gU_0 \cap S| \geq C_2|P| \text{ and } \forall \theta \in S ~ |P\cap \theta P| \geq k.
\end{equation}
By direct application of \eqref{eq: AffGroupProduct} we can see that $gU_0 = \{ (g_1, g_1 z + g_2) : z \in \C \}$. Fix an element $\theta = (g_1, g_1 z + g_2)$ in $gU_0 \cap S$. As $\theta$ is $k$-rich, since it lies in $S$, there are $k$ pairs $(p, q) \in P^2$ such that 
$$q=\theta \cdot  p = (g_1, g_1 z + g_2)\cdot p = g_1p + g_1z + g_2.$$ 
Thus, for each such pair $(p,q)$, we have
\begin{equation}\label{Eq: ActionInU_0}
    q-g_1p = g_1z + g_2.
\end{equation}

Note that $z$ is dependent only on our choice of $\theta$, not on the pair $(p,q)$ selected in $P\times \theta P$. The left-hand side lives in the set $P-g_1P$.  The set $P-g_1P$ depends only on $g$ and so is the same for all choices of $\theta\in gU_0$. This uniformity allows us to show that the $(P-g_1P)$-energy (defined below) is large. We will prove this, then show why this suffices for the claimed additive energy bound.

For a complex number $z$ in $P-g_1P$, we define its number of realisations as
$$r_{(P-g_1P)}(w) = |\{(p,q)\in P^2 : q-g_1p = w\}|.$$
By \eqref{Eq: ActionInU_0}, each of the $k$ pairs $(p,q)$ associated to $\theta$ contribute to $r_{(P-g_1P)}(g_1z+g_2)$, thus
\begin{equation}\label{eq: ManyActionReps}
    r_{(P-g_1P)}(g_1z + g_2)\geq k.
\end{equation}
We define the energy $E_{+}(P, g_1P)$ as
$$ E_+(P,g_1P) = |\{ (p,q,p',q') \in P^4 : q-g_1p = q'-g_1p'\}|. $$
We can count this energy, similarly to the calculation proceeding Lemma \ref{Lem: ManyVeryRichActions}, as
$$ E_{+}(P, g_1P) = \sum_{w}r_{(P-g_1P)}^2(w).$$
Recall that, by assumption, $|gU_0 \cap S| \geq C_2|P|$. So we have at least $C_2|P|$ choices of $\theta$. We showed in \eqref{eq: ManyActionReps} that each such $\theta$ gives a $w=g_1z+g_2$, where $z$ depends on $\theta$, such that $r_{(P-g_1P)}(w) \geq k$. Thus, recall that $k=C_1|P|$, we have
$$ E_{+}(P, g_1P) = \sum_{w}r_{(P-g_1P)}^2(w) \geq C_2|P|\cdot k^2 = C_2C_1^2|P|^3.$$
To change this into a bound on the additive energy as defined, we use Cauchy-Schwarz. Indeed,
\begin{align*}
    E_+(P, \alpha P) = \sum_{t \in P-P} r_{P-P}(t)r_{P-P}(t/\alpha) \leq \left(\sum_t r_{P-P}^2(t)\right)^{1/2} \left(\sum_t r_{P-P}^2(t/\alpha)\right)^{1/2} = E_+(P).
\end{align*}
So we have, as claimed, that
$$ C_2C_1^2|P|^3 \leq E_+(P, \alpha P) \leq E_+(P).$$

\noindent For the second statement we argue similarly. Recall, we define the torus $T(\zeta)$ as
$$T(\zeta) = \Stab(\zeta)= {\{ (h_1,h_2) \in \Aff(\C) : h_1\zeta+h_2 = \zeta \}}.$$
Our assumption on $S$ gives us that
\begin{equation}\label{eq: B'kDefinition}
    |gT(\zeta) \cap S| \geq C_3|P| \text{ and } \forall \theta \in S ~ |P\cap \theta P| \geq k.
\end{equation}
Let $h=(h_1,h_2)$ be in $T(\zeta)$ such that $\theta =  gh$ is an element in $gT(\zeta) \cap S$. As $gh$ is in $S$, it is $k$-rich. So, there are $k$ pairs $(p,q)$ in $P^2$ such that
$$ gh \cdot p = q.$$
We need the following technical lemma, which we prove directly after this result.
\begin{lemma}\label{Lem: TechnicalTorusLemma}
Let $P$ be a finite point set identified in $\C$, $\zeta$ some complex number. Suppose that $h=(h_1,h_2)\in T(\zeta)$, and $g=(g_1,g_2)\in \Aff(\C)$. 
If $p\neq \zeta$ and $gh \cdot p = q$ then
\begin{equation}\label{eq: DivisorReps}
    \frac{q-g\cdot \zeta}{p-\zeta} =g_1h_1.
\end{equation}
\end{lemma}
Notice that the shifts $q-g\cdot \zeta$ and $p-\zeta$ on the left-hand side are entirely determined by $\zeta$ and so only depend on the coset  $gT(\zeta)$.
For a complex number $\zeta'$, we define the number of representations $r_{\frac{P-g\cdot \zeta}{P-\zeta}}(\zeta')$ as
$$ r_{\frac{P-g\cdot \zeta}{P-\zeta}}(\zeta') = {\Bigg\lvert}\left\{ (p,q) \in P^2 : \frac{q-g\cdot \zeta}{p-\zeta} = \zeta'\right\}{\Bigg\rvert}.$$
The particular $k$-rich motion $gh$ determines only the right-hand side of \eqref{eq: DivisorReps}. We want that each of the $k$ pairs $(p,q)$associated to the $k$-rich motion $gh$ to give us a representation of $g_1h_1$. However, we cannot trivially rule out $\zeta \in P$. This means we have at least $k-1$ such representations (as $k$ is very large, this will not matter).


Thus, each element $gh$ in $gT(\zeta) \cap S$ gives us a $g_1h_1$ such that
\begin{equation}\label{eq: DivisionRepsInS}
    r_{\frac{P-g\cdot \zeta}{P-\zeta}}(g_1h_1) \geq k-1.
\end{equation}

We have $C_3|P|$ choices of $h$ such that $gh\in gT(\zeta) \cap S$. Each choice gives us a $\zeta'=g_1h_1$ such that \eqref{eq: DivisionRepsInS} holds.
It is possible that different choices of $gh$ give that same $g_1h_1$. Suppose we have $n_i$ repeats of $g_1h_1=\zeta_i$. We can use that $n_i(k-1) \leq (n_i(k-1))^2$, for all $n_1 \geq 1$, and that $\sum_i n_i = C_3|P|$ to obtain
\begin{align}\label{eq: MuiltEnergyCauchySchwarz}
    C_3|P|(k-1)^2 \leq \sum_{\zeta'} r_{\frac{P-g\cdot \zeta}{P-\zeta}}^2(\zeta').
\end{align}
Writing this sum here as the energy
$$ {\Bigg\lvert}\left\{(p,q,p',q') \in (P\setminus\{\zeta\})^4 : \frac{q-g\cdot \zeta}{p-\zeta}=\frac{q'-g\cdot \zeta}{p'-\zeta}\right\}{\Bigg\rvert},$$
then rearranging the division one has that
$$ \sum_{\zeta'} r_{\frac{P-g\cdot \zeta}{P-\zeta}}^2(\zeta') = \sum_{\zeta''} r_{\frac{P-g\cdot \zeta}{P-g\cdot \zeta}}(\zeta'') r_{\frac{P- \zeta}{P- \zeta}}(\zeta'').$$
Applying Cauchy-Schwarz and combining with \eqref{eq: MuiltEnergyCauchySchwarz} gives
$$ C_3|P|(k-1)^2 \leq \left(\sum_{\zeta''} r_{\frac{P-g\cdot \zeta}{P-g\cdot \zeta}}^2(\zeta'')\right)^{1/2} \left(\sum_{\zeta''}r_{\frac{P- \zeta}{P- \zeta}}^2(\zeta'')\right)^{1/2}. $$
Using the same energy redefinition trick as above converts both of these division energies into the usual multiplicative energies. So, we have that
 $$C_3|P|(k-1)^2 \leq E_\times^{1/2}(P-g\cdot \zeta)E^{1/2}_\times(P-\zeta).$$

So for translates of $P$ by either $t=\zeta$ or $t=g\cdot \zeta$ we have that
$$ E_\times(P-t)\geq C_3|P|(k-1)^2 = C_3|P|(C_1|P|-1)^2.$$
\end{proof}
\noindent We now prove Lemma \ref{Lem: TechnicalTorusLemma}, which follows from a calculation.
\begin{proof}[Proof of Lemma \ref{Lem: TechnicalTorusLemma}]
Using \eqref{eq: AffGroupProduct} we see that
$$ gh = (g_1h_1, g_1h_2 + g_2).$$
By assumption $gh \cdot p = q$. Using \eqref{eq: GroupAction} we calculate the action of $gh$ on $p+\zeta$
\begin{align*}
    gh \cdot (p+\zeta) &= g_1h_1(p + \zeta) + g_1h_2 + g_2, \\
    &= g_1h_1p + g_1(h_1\zeta+h_2) + g_2.
\end{align*}
Using that $h\in T(\zeta)$, we have that $h_1\zeta+h_2 = \zeta$. So
\begin{align*}
   gh \cdot (p+\zeta) &= g_1h_1p + g_1\zeta + g_2,\\
    &= g_1h_1p + g\cdot \zeta.
\end{align*}
We reevaluate $gh \cdot (p+\zeta)$ to obtain
\begin{align*}
    gh \cdot (p+\zeta) &= g_1h_1(p + \zeta) + g_1h_2 + g_2, \\
    &= (g_1h_1p + g_1h_2 + g_2) + g_1h_1\zeta.
\end{align*}
Using the definition of the action, we can see that $(g_1h_1p + g_1h_2 + g_2)=gh\cdot p$. So, combining this with our assumption that $q=gh\cdot p$, we have
\begin{align*}
     gh \cdot (p+\zeta) &= gh\cdot p + g_1h_1\zeta \\ 
    &= q + g_1h_1\zeta.
\end{align*}
Setting these two different evaluations equal gives
\begin{align*}
    g_1h_1p+g\cdot \zeta &= q + g_1h_1\zeta\\
    g_1h_1(p-\zeta) &= q- g\cdot \zeta
\end{align*}
Thus, as $p\neq \zeta$, we have
$$\frac{q-g\cdot \zeta}{p-\zeta} =g_1h_1.$$
\end{proof}

\section{Structure in $P$}\label{Sec: StructureInP}
We now use either the coset structure established in Corollary \ref{Coro: LargeCosetIntersection} or the additive energy bound found in Corollary \ref{Coro: FewTrianglesLargeEnergy} to give explicit geometric structure.

We start with the case when $S$ has large intersection with a coset of the type $gT(z)$. From our energy calculations above, we refer to this as the multiplicative case. When $S$ has a large intersection with a coset $gU_0$ this is referred to as the additive case and is dealt with in Section \ref{Sec: AdditiveStructure}.

\subsection{Multiplicative Case: Circle Structure}\label{Sec: MultiplicativeStructure}

By examining the structure of cosets associated with the multiplicative energy, i.e. those of the form $gT(z)$, we can deduce a stronger structural result without needing to use any additive combinatorial machinery. 
Before we do this we clarify that, as our intersection of $gT(z)$ with $\SE$ is non-empty (as it contains some of $S$), we can view the intersection $(gT(z)) \cap S$ as living in a coset of $T(z)\cap \SE$. We will denote this coset by $\theta(T(z) \cap \SE)=(gT(z))\cap \SE$. This is justified in the following lemma.

\begin{lemma}\label{Lem: CosetInSE}
Let $S$ be a subset of $\SE$, $z \in \mathbb{C}$ and $gT(z)$ a coset in $\Aff(\mathbb{C})$. If $(gT(z)) \cap S$ is non-empty, then there is some $\theta \in S$ such that $(gT(z)) \cap S \subseteq \theta(T(z) \cap \SE)$.
\end{lemma}
\begin{proof}
    It suffices to prove that $(gT(z))\cap \SE=\theta(T(z) \cap \SE)$ for some $\theta \in S$.
    For two subgroups $H_1,H_2$, the non-empty intersection of two cosets $g_1H_1,g_2H_2$ is a coset of $H_1 \cap H_2$. 
    
    As $S$ has non-empty intersection with $gT(z)$ and is contained in $\SE$, $S$ has non-empty intersection with $(gT(z))\cap\SE$. Thus, by the above, $(gT(z))\cap\SE$ must lie in a coset of $T(z) \cap \SE$. There is at least one element of $S$ in this coset, so we can choose some element of $S$ as its coset representative.
\end{proof}
\noindent By identifying any point $(x,y) \in \R^2$ with $x+iy \in \C$, we can prove the following.
\begin{proposition}\label{Prop: LargeTorusCapGivesRichCircle}
    Let $P$ be a point set in $\R^2$. Let $z\in\C$, $\theta = (e^{i\alpha},t)\in \SE = S^1 \ltimes \C$. Suppose we have a set $S\subseteq\SE$ such that $|\theta(T(z) \cap \SE)\cap S| \geq C_2|P|$ and each element of $S$ is at least $(C_1|P|)$-rich with respect to the points in $P$. Then there exists a positive constant $C_3$ and a circle $\gamma$ such that $|P\cap \gamma| \geq C_3|P|.$
\end{proposition}
\noindent The proof of Proposition \ref{Prop: LargeTorusCapGivesRichCircle} relies on the following technical Lemma.

\begin{lemma}\label{Lem: OrbitsAreCircles}
Fix some $p$ and $z$ in $\R^2$ and $\theta=(e^{i\alpha}, t)$ in $\SE = S_1 \ltimes \C$. The orbit of $p$ under the action of the coset $\theta (T(z)\cap \SE)$ is a circle in $\R^2$, the radius is a function of $z$ and $p$ only.
\end{lemma}

Lemma \ref{Lem: OrbitsAreCircles} follows from the observation that intersecting the stabilisers $T(z)$ with $\SE$ reduces the dimension of $T(z)$ from two, as a subset of the affine group, to one in $\SE$. This dimension reduction allows us to show that we have a constant proportion of our points on a circle. The proof of Lemma \ref{Lem: OrbitsAreCircles} is technical, so we will delay it until after Proposition \ref{Prop: LargeTorusCapGivesRichCircle}.

\begin{proof}[Proof of Proposition \ref{Prop: LargeTorusCapGivesRichCircle}]
Define a digraph with vertex set \(P\) and directed edges \((p\to q)\) if there exists an element of $\theta T(z)$ taking \(p\) to \(q\). By assumption there are $ C_2|P|$ elements of $\theta T(z)$ each at least $C_1|P|$-rich, and so there are \(C_1C_2 |P|^2\) edges in this graph. By the handshaking lemma and the pigeon-hole principle there is a vertex of out-degree $C_3|P|$, for a positive constant $C_3$. 

Note that if $(p\to q)$ is an edge then $q$ must lie in the orbit of $p$ under the action of the coset (as this is how edges are defined), so all the $C_3|P|$ points of $P$ that are out-neighbours of $p$ lie on this orbit. By Lemma \ref{Lem: OrbitsAreCircles} this orbit is a circle, so we are done.
\end{proof}

\noindent We now prove Lemma \ref{Lem: OrbitsAreCircles}. This directly uses the structure of the group $T(z)\cap \SE $.

\begin{proof}[Proof of Lemma \ref{Lem: OrbitsAreCircles}]
Recall from the embedding \eqref{eq: Embedding2} that a translation in $\SE$ is an element of the form $(1,t)$ for $t\in \C$. We claim that every coset of the from $\theta(T(z) \cap \SE)$ contains exactly one translation. The proof now follows quickly.

Suppose that $\theta$ is the identity, then our coset $\theta(T(z) \cap \SE)$ is exactly the set of rigid motions that preserve $z$. These are the rotations, in $\R^2$, about the point $z$ (as we are in $\SE$ and not $\Aff(\C)$ we do not have any scalings). So, if $p$ is a point, its orbit under the action of $T(z)\cap\SE$ will be the circle of radius $||p-z||$ centred at $z$.

If $\theta$ is not the identify, relabel so that $\theta$ is the unique translation in $\theta(T(z)\cap\SE)$. The orbit of $p$ under $T(z)\cap\SE$ is a circle and we are applying the uniform translation $\theta$ to this orbit, so it remains a circle. One can see that this circle is centered at $\theta\cdot z$ with radius $||p-z||$.

We now need to justify why each coset $\theta (T(z) \cap \SE)$ has exactly one translation. 
By definition the group $T(z)\cap \SE$ are the rigid motions that preserve the point $z$. As this is in $\SE$, these are exactly the rotations about $z$. So, using the embedding \eqref{eq: Embedding1}, we can write these rotations as
$$ T(z)\cap\SE = \{ (e^{i\vartheta}, z(1-e^{i\vartheta})) : \vartheta \in [0,2\pi) \}. $$
Taking $\theta=(e^{i\alpha},t)$ in $\SE$ and using the group multiplication law in \eqref{eq: AffGroupProduct}, we see that
$$ (e^{i\alpha},t)(T(z)\cap\SE) =  \{ (e^{i(\vartheta+\alpha)}, ze^{i\alpha}(1-e^{i(\vartheta)}) + t) : \vartheta \in [0,2\pi)\}.$$
By embedding \eqref{eq: Embedding2}, translation requires the first component to be $1\in S^1$. This happens for exactly one value of $\vartheta$ i.e. when $\vartheta = -\alpha$. So $\theta (T(z) \cap \SE)$ has exactly one translation.
\end{proof}

\subsection{Additive Case: Line Structure}\label{Sec: AdditiveStructure}

Our aim in this section is to demonstrate that a point set with few classes of congruent triangles has: large intersection with a line $l$; a positive proportion of $P$ lies on very-rich lines all parallel to $l$. We prove the following theorem, essentially a rephrasing of \cite[Lemma 3.2]{Akshat19}.
\begin{theorem}\label{Theorem: PointLineStructure}
Suppose that $P$ is a point set in $\R^2$ with $E_+(P)\geq C|P|^3$ then there is a subset $P'$ of $P$ with $|P'|\geq C_1|P|$ and there exist parallel lines $l_1, l_2 ,\ldots, l_r$ in $\R^2$, and constants $0 < \sigma \leq 1$, $C_3, C_4>0$ such that
$$ |P' \cap l_1| \geq \cdots \geq |P' \cap l_r| \geq \frac{|P' \cap l_1|}{C_3} \geq C_4|P|^\sigma,  $$
and
$$ |P'\setminus (l_1\cup \cdots \cup l_r)|<|P'|/2. $$
\end{theorem}

We do not use the structure of the affine group and instead rely only on our energy bound. Because of this reliance, we need $P$ to have essentially maximal additive energy. 

In order to discuss this, we need to introduce the notion of a sumset. For a subset $A$ in an abelian group $G$ we define the sumset $A+A$ to be all distinct pairwise sums,
$$ A + A  = \{ a + a' : a,a' \in A\}.$$
The sumset is closely related to the additive energy, see \eqref{eq: SumSetToEnergy}. To see why an essentially maximal energy bound is required, we introduce two results. For a given $\varepsilon>0$, Stanchescu \cite[Theorem 5(b)]{stanchescu2002planar} and Erd\H{o}s-F\"uredi-Pach-Ruzsa \cite[Theorem 3.1]{gridRevisited} find Behrend-type point sets $P \subset \R^2$ with no collinear triples such that $|P+P|\leq c|P|^{1+\varepsilon}.$ A similar Cauchy-Schwarz argument to those in Section \ref{Sec: FewTrianglesMAnyRichActions} shows
\begin{equation}\label{eq: SumSetToEnergy}
    |P+P|\leq c|P|^{1+\varepsilon} \Rightarrow E_+(P)\geq |P|^{3-\varepsilon}.
\end{equation}
Fortunately, Corollary \ref{Coro: FewTrianglesLargeEnergy} shows that point sets with few classes of congruent triangles have the essentially maximal energy $E_+(P) \geq c|P|^3$. 

In the proof of Theorem \ref{Theorem: PointLineStructure} we use traditional tools of additive combinatorics. The Balog-Szemer\'edi-Gowers Theorem \cite{balog1994BalogSzGowers, gowers2001newSzemeredi} allows us to pass form additive energy to additive sets. We then borrow Mudgal's \cite{Akshat19} use of Green and Ruzsa's \cite{green2007freiman} Freiman-type Theorem to provide the line structure. We will introduce both these results below.

The use of energy in the additive case, as opposed to the use of the affine group structure in the multiplicative case, is the reason for the different strength in the two structural results. One can see that the unipotent subgroup $U_0$ has the same dimension in both $\Aff(\C)$ and $\SE$. So, our dimension-reduction trick used in Section \ref{Sec: MultiplicativeStructure} no longer works. One way to improve the structure found in Theorem \ref{Thm: MainStructuralResult} would be to find a way to use the affine group structure.

\subsubsection{Additive Tools}
We introduce the tools we will use from additive combinatorics. 
The following is the version of Balog-Szemer\'edi-Gowers Theorem we will use. We are not concerned with the explicit constants, as they do not effect the exponent in Theorem \ref{Thm: MainStructuralResult}. We refer those interested in the explicit constants to the proof of Balog \cite{balog2007manyadditive}.

\begin{theorem}[Balog-Szemer\'edi-Gowers]\label{Thm: BSG} Let $A$ be a finite set in an abelian group $G$ with respect to the group operation 
$+$. Then if $E_+(A)\geq \alpha|A|^3$ 
we have an $A' \subseteq A$ such that
$$ |A'| \geq C_1|A| \text{ and } |A' + A'| \leq C_2|A|, $$
where $C_1$ and $C_2$ are positive constants dependent on $\alpha$ only.
\end{theorem}
We use the Freiman-type result of Green--Ruzsa. In order to state the result we have to give some additional definitions.

Given an abelian group $G$, we define a proper progression $\mathbb{P}$ of arithmetic dimension $s$ and size $\Lambda$ as
$$\mathbb{P} = \{v_0 + u_1v_1 + \cdots + u_sv_s | 0 \leq u_i < \lambda_i \text{ for } 1 \leq i \leq s\},$$
where $\lambda_1\cdots\lambda_s = \Lambda$ and $v_0, v_1,\ldots, v_s$ are elements of $G$ such that all the sums in the progression are distinct.

We further define a coset progression to be a set of the form $\mathbb{P} + H$ where $\mathbb{P}$ is a proper progression and $H$ is a subgroup of $G$.

\begin{theorem}[Green--Ruzsa \cite{green2007freiman}]\label{Thm: GreenRuzsa}
Let $A$ be a subset of an abelian group G such that $|A + A| \leq
K|A|$. Then $A$ is contained in a coset progression of arithmetic dimension $s \leq C_5K^4
\log(K + 2)$ and size $\Lambda = |\mathbb{P} + H| \leq e^{C_5K^4log^2(K+2)}|A|$, for some constant $C_5 > 0$.
\end{theorem}

\subsubsection{Proof of Theorem \ref{Theorem: PointLineStructure}}

The proof of Theorem \ref{Theorem: PointLineStructure} is--but for some constants changing--the same as the Mudgal's proof \cite[Lemma 3.2]{Akshat19}. We include it here for completeness.
%
%
\begin{proof}[Proof of Theorem \ref{Theorem: PointLineStructure}]
Use Balog-Szemer\'edi-Gowers on $P$ to gain a subset $P'$ such that $|P'|\geq C_1|P|$ and $|P'+P'| \leq \frac{C_2}{C_1}|P'|$. From this point we are replicating the proof of \cite[Lemma 3.2]{Akshat19}.

Applying Green--Ruzsa to $P'\subseteq \R^2$ we have that $P'$ lies in a coset progression $\mathbb{P} + H$.
Note that, as $P'$ is finite, the bound on $\Lambda$ implies $H$ must be a finite subgroup. The only such subgroup over $\R^2$ is the trivial group. Thus, we can assume that $P'$ lies in the proper progression $\mathbb{P}$ only.

So, we have that $P' \subseteq \mathbb{P}$ where $\mathbb{P}$ is of size $\Lambda \leq C_6|P|$ and arithmetic dimension $s \leq C_7$,
$$ \mathbb{P} = \{v_0 + u_1v_1 + \cdots + u_sv_s | 0 \leq u_i < \lambda_i \text{ for } 1 \leq i \leq s\}. $$
By definition of the progression we have that $\lambda_1 \cdots \lambda_s = \Lambda \leq C_6|P|$, so by the pigeon-hole principle there is an $i$ such that $\lambda_i \geq C_6^{1/s}|P|^{1/s}$, without loss of generality let $i=1$. Define the arithmetic progression $Q$ as
$$Q = \{ u_1v_1 : 0 \leq u_1 <\lambda_1\}.$$
The key step of \cite[Lemma 3.2]{Akshat19} is to think of the proper progression $\mathbb{P}$ as a set of $\Lambda/\lambda_1$ translates of $Q$. As $\mathbb{P}$ is proper these translates are disjoint and thus
$$ \frac{\Lambda}{\lambda_1} < \frac{C_6|P|}{C_6^{1/s}|P|^{1/s}} = C_6^{1-1/s}|P|^{1-1/s}. $$
We have $C_6^{1-1/s}|P|^{1-1/s}$ translates of $Q$ covering all $|P'|\geq C_1|P|$ points of $P'$ and thus by the pigeon-hole principle there is a translate that contains at least $C_1C_6^{1/s-1}|P|^{1/s}$ points of $P'$. This gives us our rich line, which we will denote as $l_1$.

We now examine the above cover to show that a positive proportion of $P'$, and hence $P$, is covered by the translates of $l_1$. Order the cover so that
$$|P' \cap l_1| \geq \cdots \geq |P' \cap l_k| > 0,$$
where each $l_k$ is a translate of $Q$ with non-empty intersection with $P'$. Let $C_3$ be a constant to be determined later, then let $r$ be the minimal positive integer such that
$$|P' \cap l_1| \geq \cdots \geq |P' \cap l_r| \geq \frac{|P' \cap l_1|}{C_3} = \frac{C_1C_6^{1/s-1}|P|^{1/s}}{C_3},$$
if $C_3 \geq 1$ such an $r$ exists. Define the set $B$, of points in $P'$ our rich lines miss, as \linebreak $B=P'\setminus \bigcup_{i=1}^r(P'\cap l_i)$. To prove the result it suffices to show $|B|\leq |P'|/2$. We have that
\begin{equation}\label{Eq: BEq1}
    |B| = \sum_{j=r+1}^k |P'\cap l_j| < (k-r)\frac{|P'\cap l_1|}{C_3},
\end{equation}
rearranging gives $(k-r) > \frac{|B|C_3}{|P'\cap l_1|}$. We also have, by selecting relevant subsets of $P'$ and using that the lines $l_j$ are disjoint, that
\begin{equation}\label{Eq: BEq2}
     |P'+P'| \geq |(P'\cap l_1) + B | = \sum_{j=r+1}^k |(P'\cap l_1) +  (P'\cap l_j)| \geq (k-r)|P'\cap l_1|.
\end{equation}
Combining \eqref{Eq: BEq1} and \eqref{Eq: BEq2} we have that
$$  |P'+P'| > |B|C_3. $$
Using the constant $C_2$ from Theorem \ref{Thm: BSG} we have that
$$|B| < \frac{|P'+P'|}{C_3} \leq \frac{C_2|P'|}{C_3}. $$
Choosing $C_3$ to be the maximum of $2C_2$ and $1$ we can ensure that $|B|<|P'|/2$. To finish the proof as stated let $\sigma = 1/s$ and $C_4 = \frac{C_1C_6^{1/s-1}}{C_3}$.
\end{proof}

\subsection{Proof of Theorem \ref{Thm: MainStructuralResult}}
We now have all the necessary tools to prove our main theorem.
\begin{proof}[Proof of Theorem \ref{Thm: MainStructuralResult}]
We start with Corollary \ref{Coro: LargeCosetIntersection}. This gives us an $S$ in $\SE$, with each element of $S$ a rigid motion that is at $(C_1|P|)$-rich when acting on $P$. Further, there is some $g$ in $\Aff(\C)$, so that at least one of the following holds:
\begin{enumerate}
    \item $\displaystyle |gU_0\cap S| \geq C_2 |P|$,
    \item Some $z\in \C$ so that $\displaystyle |gT(z)\cap S| \geq C_3 |P|$.
\end{enumerate}
If the first of these holds, we get the line structure by combining the first statement of Proposition \ref{Prop: LargePEnergy} with Theorem \ref{Theorem: PointLineStructure}.

If the second holds, note that $S$ has non-empty intersection with a coset $gT(z)$ and $S$ is contained in $\SE$. So we can apply Lemma \ref{Lem: CosetInSE} then Proposition \ref{Prop: LargeTorusCapGivesRichCircle}.
\end{proof}

\bibliography{DiscreteGeometryReferences}{}

\bibliographystyle{plain}

\bigskip
		\footnotesize
		
		 \noindent\textit{E-mail address}, S.~Mansfield: \texttt{sam.mansfield@bristol.ac.uk}\\
		 \textit{E-mail address}, J.~Passant: \texttt{jonathan.passant@bristol.ac.uk}\\
		
		\noindent\textsc{Department of Mathematics, University of Bristol, Bristol, BS8 1UG}\par\nopagebreak

\end{document}